\documentclass[11pt]{amsart}

\usepackage{amsmath}
\usepackage{amsfonts}
\usepackage{amssymb}
\usepackage{amsthm}
\usepackage{graphics}
\usepackage{amscd}
\usepackage{graphicx,epsfig}
\usepackage{color}
\usepackage[all]{xy}
\usepackage{tikz-cd}
\usepackage{url}

\DeclareMathOperator{\tr}{tr}
\DeclareMathOperator{\Div}{div}

\renewcommand{\epsilon}{\varepsilon}

\newcommand{\boF}{\mathcal{F}}

\newcommand{\boM}{\mathcal{M}}

\newcommand{\R}{\mathbb{R}}

\newcommand{\Z}{\mathbb{Z}}

\newcommand{\B}{\mathbb{B}}
\newcommand{\D}{\mathbb{D}}

\renewcommand{\S}{\mathbb{S}}

\newcommand{\la}{\langle}
\newcommand{\ra}{\rangle}
\newcommand{\eps}{\varepsilon}
\newcommand{\id}{\text{id}}

\newcommand{\Ome}{\Omega}

\newtheorem{thm}{Theorem}

\newtheorem{prop}[thm]{Proposition}
\newtheorem{lem}[thm]{Lemma}
\newcommand{\inter}[1]{\overset{\circ}{#1}}
\newcommand{\barre}[1]{\overline{#1}}
\renewcommand{\phi}{\varphi}
\DeclareMathOperator{\Int}{int}

\newtheorem*{thm*}{Theorem}

\newtheorem*{claim*}{Claim}

\theoremstyle{remark}

\newtheorem{remarq}{Remark}
\newtheorem*{rem*}{Remark}

\newcounter{remark}

\newcounter{case}

\newcounter{construction}

\newcounter{fact}

\newcounter{step}

\newcommand{\wPhi}{\widetilde\Phi}
\DeclareMathOperator{\diff}{Diff}

\title{Minimal planes in asymptotically flat three-manifolds}
\author{Laurent Mazet}
\address{Universit\'e Paris-Est, LAMA (UMR 8050), UPEC, UPEM, CNRS, 61, avenue
du G\'en\'eral de Gaulle, F-94010 Cr\'eteil cedex, France}
\email{laurent.mazet@math.cnrs.fr}

\author{Harold Rosenberg} 
\address{Instituto Nacional de Matematica Pura e Aplicada (IMPA) Estrada Dona
Castorina 110, 22460-320, Rio de Janeiro-RJ, Brazil}
\email{rosen@impa.br}

\begin{document}

\maketitle

\begin{abstract}
In this paper, we improve a result by Chodosh and Ketover \cite{ChKe}. We prove
that, in an asymptotically flat $3$-manifold $M$ that contains no closed minimal
surfaces, fixing $q\in M$ and $V$ a $2$-plane in $T_qM$ there is a properly
embedded minimal plane $\Sigma$ in $M$ such that $q\in\Sigma$ and $T_q\Sigma=V$. We also prove
that fixing three points in $M$ there is a properly embedded minimal plane
passing through these three points.
\end{abstract}


\section{Introduction}


This paper is inspired by the recent work of Otis Chodosh and Daniel
Ketover~\cite{ChKe}. They consider an asymptotically Euclidean $3$-manifold $M$
that contains no closed minimal surfaces. Then they prove the existence of
a properly embedded minimal plane passing through any point $q$ of $M$.

Here we obtain some more information:
\begin{itemize}
\item if $q\in M$ and $V\subset T_qM$ is a $2$-plane, then there is a properly
embedded minimal plane $\Sigma$ in $M$ with $q\in M$ and $T_q\Sigma=V$.
\item if $q_1,q_2,q_3$ are $3$ points of $M$, there is a properly embedded
minimal plane $\Sigma$ in $M$ containing the three points.
\end{itemize}

The main ingredient is the following Theorem
\begin{thm*}
Let $\B$ be the unit ball of $\R^3$ and let $M$ be $\B$ with a Riemannian metric
$g$ satisfying
\begin{itemize}
\item $\partial M$ is strictly mean convex, and
\item $M$ contains no closed minimal surfaces.
\end{itemize}
Then we have the following statements.
\begin{enumerate}
\item If $q\in \Int (M)$ and $V\subset T_qM$ is a $2$-plane, there is an
embedded minimal disk $D\in M$ with $\partial D=D\cap \partial M$, $q\in D$ and
$T_qD=V$.
\item If $q_1,q_2,q_3$ are $3$ points in $\Int(M)$, there is an embedded minimal
disk $D$ in $M$ with $\partial D=D\cap \partial M$ and $q_1,q_2,q_3$ in $D$.
\end{enumerate}
\end{thm*}

The introduction to the paper of Chodosh and Ketover~\cite{ChKe} contains an
excellent exposition of the problem of the existence of minimal surfaces in
asymptotically flat $3$-manifolds, in particular, properly embedded minimal
planes. They discuss the relevance of this theory to general relativity. They also
explain why they employ the theory of Brian White "moduli space of minimal
surfaces and degree theory" to prove their theorem rather than variational
techniques, solving Plateau problems in geodesic balls and taking limits.

Our proof of our main Theorem above also starts with the theory of White, the
moduli space of embedded minimal disks in the ball $M=(\B,g)$.

A natural question is what are the complete Riemannian $3$-manifolds that
contain properly embedded minimal planes and what are such planes. In the
Euclidean $\R^3$, we know the only such planes are affine planes and helicoids
\cite{MeRo3}. Moreover affine planes are completely characterized by a point in it and
its tangent plane at that point or by $3$ non collinear points contained in it. This explains
why we are interested in prescribing these constraints. We notice that in the
case of the Euclidean $\R^3$ our proof produces affine planes and no helicoids.

Actually, in the general case, the minimal planes we construct have quadratic
area growth. One might suspect that these minimal planes are the
only ones with these constraints with quadratic area growth as it is the case in
$\R^3$. However this is not true; we construct a counterexample. 

For other ambient spaces, the existence of properly embedded planes is still an
open question. For example, when $g$ is a complete metric of non positive
curvature, we do not know if such a minimal plane exists. When the curvature is
strictly pinched, by the work of Anderson they exist \cite{And}. When $S$ is a
closed surface, there is no properly embedded minimal plane in $S\times \R$, with the
product metric.
The same problem for embedded
minimal annuli $\S^1\times \R$ is also very interesting: is it possible to
construct "catenoids" in some ambient manifolds? In the same spirit, we mention a 
question:  Let $M$ be a complete non compact Riemannian $3$-manifold.  Is there a 
properly embedded minimal surface in $M$?

The paper is organized as follow. In Section~\ref{sec:white}, we recall the main
points of White's theory. Section~\ref{sec:prescrib} is devoted to the proof of item 1
of the above main theorem (Theorem~\ref{th:prescrip}). Section~\ref{sec:3point}
deals with item 2 of the main theorem (Theorem~\ref{th:3point}). In 
Section~\ref{sec:extent}, we explain how one can pass from a minimal disk in a compact 
ball to a
properly embedded minimal plane in an asymptotically Euclidean space, here we
use ideas developed by Chodosh and Ketover in \cite{ChKe}. In the last section, we give a 
counterexample to some uniqueness statement.

The first author would like to thank P. Topping for very interesting discussions.


\section{The space of minimal disks}
\label{sec:white}


After Tomi and Tromba~\cite{ToTr}, White~\cite{Whi} has developed the theory of
the moduli space of minimal submanifolds. Here we recall some of the aspects of
this theory that we will use.

Let us denote by $\B$ the closed unit ball in $\R^3$ and $\D$ the closed unit
disk in $\R^2$. Let $M$ be $\B$ endowed with a Riemannian metric $g$.

Let $f_i:\D\to \R^3$ ($i=1,2$) be two $C^{k,\alpha}$ maps ($k\ge 2$ and $0<\alpha<1$). They are said to be
equivalent if $f_1=f_2\circ \phi$ for some $C^1$ diffeomorphism $\phi$ of $\D$
that is the identity on $\partial \D=\S^1$ ;
let $\boF$ denote the set of equivalence classes. We also denote by
$\diff_{k,\alpha}(\D,\partial \D)$ the set of $C^{k,\alpha}$ diffeomorphisms of $\D$ that are the
identity on $\partial \D$.

 On $\boF$, there is a boundary
operator $\delta:\boF\to C^{k,\alpha}(\S^1, \R^3)$ which is defined by
$\delta([f])=f_{|\S^1}$. The space of immersed minimal disks $\widetilde \boM$
is then the space of equivalence classes $[f]$ of $C^{k,\alpha}$ minimal
immersions $f:\D\to M$ such that $\delta ([f])\in C^{k,\alpha}(\S^1,\partial
M)$ and $f$ is never tangent to $\partial M$ (see Section~8 in~\cite{Whi}). We
notice that if $[f_1]=[f_2]$ where $f_1$ and $f_2$ are two $C^{k,\alpha}$
immersions then $\phi\in \diff_1(\D,\partial \D)$  such that $f_1=f_2\circ \phi$
is actually in $\diff_{k,\alpha}(\D,\partial \D)$ (see Appendix~\ref{appendix}).

The main result of~\cite{Whi} is the following theorem.
\begin{thm}
The space $\widetilde\boM$ is a smooth Banach manifold and 
$$
\delta:\widetilde\boM\to C^{k,\alpha}(\S^1,\partial M)
$$
is a smooth Fredholm map of index $0$.
\end{thm}

The descriptions of the charts of $\widetilde\boM$ are given in
\cite[Theorem~3.3]{Whi} (see also the proof of Lemma~\ref{lem:smale}).

Actually, we are going to consider just the open subset $\boM$ of $[f]\in\widetilde
\boM$ where $f$ is a minimal embedding of the disk. $\boM$ is the space of
minimal embedded disks in $M$ with boundary in $\partial M$ and never tangent to $\partial M$.

A second important result gives conditions under which the map $\delta$ is proper.
\begin{thm}\label{th:properness}
Let us assume that $M$ satisfies the following two conditions
\begin{itemize}
\item $\partial M$ is strictly mean-convex towards $M$ and
\item $M$ contains no closed minimal surfaces. 
\end{itemize}
Then the map $\delta:\boM \to C^{k,\alpha}(\S^1,\partial M)$ is proper
\end{thm}

This result comes by combining different results by White: Theorem~C in \cite{Whi2} gives
the properness thanks to the area estimate of Theorem~2.1 in~\cite{Whi4} ($M$
has mean convex boundary and no closed embedded minimal surfaces) and the curvature
estimate given by Theorem~0 in~\cite{Whi3}.


\section{Prescribing the tangent plane to a minimal disk in a Riemannian $3$-ball}
\label{sec:prescrib}


In this section we want to find a minimal disk in $M$ passing through a point $q$
with a prescribed tangent space. We still assume that $M$ is $\B$ endowed with a
metric $g$ but we also assume that $g$ can be extended to the whole $\R^3$. Thus
the main result of this section is the
following theorem

\begin{thm}\label{th:prescrip}
Let $M=(\B,g)$ be as above such that
\begin{itemize}
\item[(H1)] $\partial M$ is strictly mean convex towards $M$ and
\item[(H2)] $M$ contains no closed minimal surfaces. 
\end{itemize}
Let us consider $q\in M\setminus \partial M$ and $V$
be a $2$-dimensional subspace of $T_qM$. There exists $\sigma\in \boM$ such that
$q\in \sigma$ and $V=T_q\sigma$.
\end{thm}

The remaining part of this section is devoted to the proof of this theorem.

Let us describe the boundary curves that we will consider. The unit tangent bundle
to $\S^2$ is $U\S^2=\{(p,v)\in \S^2\times \S^2|\la p,v\ra =0\}$. For $t\in(-1,1)$,
$(p,v)\in U\S^2$ and $e^{i\theta}\in\S^1$, we define
$P(p,v,t,e^{i\theta})=tp+\sqrt{1-t^2}(\cos \theta v+\sin\theta p\wedge v)$. When
$e^{i\theta}$ runs along $\S^1$, this describes the parallel circle
$\{q\in\S^2|\la p,q\ra=t\}$.

Thus we obtain a map:
\begin{equation}\label{eq:Phi}
\Phi: \begin{matrix}
U\S^2\times(-1,1)&\longrightarrow &C^{k,\alpha}(\S^1,\S^2)\\
(p,v,t)&\longmapsto &\zeta\mapsto P(p,v,t,\zeta)
\end{matrix}
\end{equation}

Actually $\Phi$
gives the whole family of parallel curves on $\S^2$. By changing $v$, we only change
the origin of the parametrization.

Basically, the idea of the proof of Theorem~\ref{th:prescrip} consists in considering all the
minimal disks in $M$ bounded by a parallel $\Phi(p,v,t)$ and then finding a
perturbation of one of them that satisfies the desired property.

\subsection{Study when $|t|$ is close to $1$}


In this section we prove than when $|t|$ is close to $1$ there is only one
minimal disk in $M$ bounded by $\Phi(p,v,t)$. More precisely we have the
following statement.

\begin{prop}\label{prop:closepole}
There is $\bar\eps>0$ such that, for any $(p,v)\in U\S^2$ and $1-\bar\eps\le|t|<1$, there
is a unique minimal disk $D_{p,v,t}$ in $M$ bounded by $\delta (D_{p,v,t})=\Phi(p,v,t)$.

Moreover, when $(p,v)$ are fixed, the $D_{p,v,t}$ form a foliation for
$t\in[1-\bar\eps,1)$ and $t\in(-1,-1+\bar\eps]$ by stable non-degenerate minimal disks.
$D_{p,v,t}$ can be parametrized by
$$
F(p,v,t): (x,y)\in \D\mapsto \sqrt{1-t^2}(xv+yp\wedge v)+(t+u_{p,v,t}(x,y))p
$$
where $(p,v,t)\mapsto u_{p,v,t}$ is a smooth family of $C^{k,\alpha}$ functions with
$u_{p,v,t}=(1-t)h_{p,v}+o(1-t)$ (we have the same formula close to $-1$) and $0$ boundary values.
\end{prop}

\begin{proof}
First let us recall that if a Riemannian metric in $\R^3$ is given by symmetric
matrices $G$ and $u$ is a function defined on a subset of $\R^2$ then its graph
is minimal if
$$
0=\frac1{\sqrt{\det G}}\Div_{\scriptscriptstyle 3}\left(\frac{\sqrt{\det
G}}{\sqrt{^t\nabla_{\scriptscriptstyle 3} f (G^{-1})\nabla_{\scriptscriptstyle 3}
f}}G^{-1}\nabla_{\scriptscriptstyle 3} f\right) _{|(x,y,u(x,y))} 
$$
where $f$ is the function on $\R^3$ defined by $f(x,y,z)=u(x,y)-z$ and
$\Div_{\scriptscriptstyle 3}$ and $\nabla_{\scriptscriptstyle 3}$ are the
Euclidean divergence and gradient operator on $\R^3$.

For $(p,v,t)\in U\S^2\times(-1,1)$ we are looking for a minimal disk that is a
graph over the Euclidean disk bounded by $\Phi(p,v,t)$. So we are looking for a
function $u_{p,v,t}$ defined on $\D$ with vanishing boundary values such that the
surface
$$
F(p,v,t): (x,y)\in \D\mapsto \sqrt{1-t^2}(xv+yp\wedge v)+(t+u_{p,v,t}(x,y))p
$$
is in $\B^3$ and is minimal for the metric $g$. Since $t$ is close to
$1$ (the same can be done for $t$ close to $-1$), we write $t=\cos\beta$
and we look for $w_{p,v,\beta}=u_{p,v,\cos\beta}$ such that
$$
\widetilde F(p,v,\beta):(x,y)\in \D\mapsto
\sin\beta(xv+yp\wedge v)+(\cos\beta+w_{p,v,\beta}(x,y))p
$$
is in $\B^3$ and is minimal for the metric $g$. For fixed $p$, $v$, $\beta$, we
define the conformal map $M_{p,v,\beta}:\D\times\R\to \R^3 ;(x,y,z)\mapsto \sin\beta(xv+yp\wedge
v)+(\cos\beta+\sin\beta z) p$. So $\widetilde F(p,v,t)$ is
minimal if the graph of $\tilde w_{p,v,\beta}=\frac{w_{p,v,\beta}}{\sin\beta}$ in
$\D\times\R$ is minimal for the metric
$g_{p,v,\beta}=\frac{M_{p,v,\beta}^*g}{\sin^2\beta}$. As explained
above if $G$ is the matrix of $g$ in $\R^3$, the graph of $\tilde w$ in $\D\times\R$ is
minimal for $g_{p,v,\beta}$ if
\begin{multline*}
0=\frac1{\sqrt{\det G_{p,v,\beta}}}\Div_{\scriptscriptstyle
3}\left(\frac{\sqrt{\det G_{p,v,\beta}}}{\sqrt{^t\nabla_{\scriptscriptstyle 3}
f (G_{p,v,\beta}^{-1})\nabla_{\scriptscriptstyle 3}
f}}G_{p,v,\beta}^{-1}\nabla_{\scriptscriptstyle 3} f\right)_{|(x,y,\tilde w(x,y))}\\
=H(p,v,\beta,\tilde w)(x,y)
\end{multline*}
where $f(x,y,z)=\tilde w(x,y)-z$ and $G_{p,v,\beta}=(v,p\wedge v,p)^*(G\circ
M_{p,v,\beta})(v,p\wedge v,p)$ where $(v,p\wedge v,p)$ is the matrix whose
columns are $v$, $p\wedge v$ and $p$.

So one can consider the map
$$
H:U\S^2\times(-\eps,\eps)\times C_0^{k,\alpha}(\D)\to C^{k-2,\alpha}(\D);(p,v,\beta,\tilde w)\mapsto
H(p,v,\beta,\tilde w)
$$
where $C_0^{k,\alpha}(\D)$ is the subset of $C^{k,\alpha}$-functions vanishing on
$\partial\D$. We want to solve $H(p,v,\beta,\tilde w)=0$.

Let us remark that for $\beta=0$, $H(p,v,0,\tilde w)$ computes the mean curvature of
the graph of $\tilde w$ in the constant metric $G_{p,v,0}= (v,p\wedge
v,p)^*G(p)(v,p\wedge v,p)$. So the function $\tilde w\equiv 0$ is a solution,
$H(p,v,0,0)=0$. Besides we have
$$
D_{\tilde w}H(p,v,0,0)(h)=\frac1{(G_{p,v,0}^{33})^{3/2}} \sum_{i,j=1}^2
(G_{p,v,0}^{ij}G_{p,v,0}^{33}- G_{p,v,0}^{i3}G_{p,v,0}^{j3})h_{ij}
$$
where the $G_{p,v,0}^{ij}$ are the coefficients of $G_{p,v,0}^{-1}$. This
operator is elliptic with constant coefficients so it is invertible from
$C_0^{k,\alpha}(\D)$ to $C^{k-2,\alpha}(\D)$ and the implicit function theorems applies.
So there is a family of functions $\tilde w_{p,v,\beta}$ ($\beta$ close to $0$) which solves
$H(p,v,\beta,\tilde w_{p,v,\beta})=0$ and $\tilde w_{p,v,0}=0$. Moreover we have
$\tilde w_{p,v,\beta}=\beta\tilde h_{p,v}+o(\beta)$.

Thus the family $(u_{p,v,t})$ is well defined for any $(p,v)\in U\S^2$ and
$t\in(1-\eps,1)$. Using that $t=\cos \beta$, we obtain
$u_{p,v,t}=(1-t)\tilde h_{p,v}+o(1-t)$. It still remains to prove that these
disks are in $\B$ and form local foliations. By (H1), we notice that in the
metric $g$ the
Euclidean spheres $\S_r^2$ of radius $r$ are mean convex for
$r\in[1-\delta,1+\delta)$. Since the $F(p,v,1)(\D)=p$, for $t$ close to $1$,
$F(p,v,t)(\D)$ is contained in the part of $\R^3$ foliated by the $\S_r^2$
($1-\delta< r<1+\delta$). Since $\partial (F(p,v,t)(\D))\subset \S^2$, the comparaison principle implies
$F(p,v,t)(\D)\subset \B$ for any $(p,v,t)\in U\S^2\times (1-\eps,1)$. 

If $t<t'<1$, the boundary curve $\Phi(p,v,t')(\S^1)$ is between
$\Phi(p,v,t)(\S^1)$ and $p$. Besides, for $t'$ close to $1$, $F(p,v,t')(\D)$ is
close to $p$ and does not intersect $F(p,v,t)(\D)$. So by the maximum principle
$F(p,v,t')(\D)$ does not intersect $F(p,v,t)(\D)$ for any $1-\eps<t<t'<1$. Thus
at fixed $(p,v)$, the minimal disks $\{F(p,v,t)(\D)\}_{1-\eps<t<1}$ form a
foliation. As a consequence, these minimal disks are stable, non degenerate and
moreover they are the unique minimal disk with boundary $\Phi(p,v,t)(\S^1)$ at
least for $1-\bar\eps\le t<1$ ($\bar\eps<\eps$). This uniqueness statement comes from the
work of White in \cite{Whi2} (the arguments start at the bottom of page 148,
notice that here we have already constructed a foliation, see also Remark~\ref{rk:uniq} below).
\end{proof}

\begin{remarq}
Actually the function $\tilde h_{p,v}$ can be computed by the equation
$0=D_\beta H(p,v,0,0)+D_{\tilde w}H(p,v,0,0)(\tilde h_{p,v})$ that can be solved
since $D_{\tilde w}H(p,v,0,0)$ is invertible. Moreover, $D_\beta H(p,v,0,0)$ can
be computed and is given by the difference of the mean curvatures at $ p$ of
$\S^2\subset (M,g)$ (which is positive) and $\S^2\subset \R^3$ endowed with the
constant metric $G(p)$ (which is also positive). So $D_\beta H(p,v,0,0)$ is a
constant function on $\D$ and $\tilde h_{p,v}(x,y)=\lambda (1-(x^2+y^2))$ for
some $\lambda$.
\end{remarq}

\begin{remarq}\label{rk:uniq} For the readers benefit, we give the arguments of White
concerning uniqueness. So we fix $(\bar p,\bar
v)\in U\S^2$. We know that for any $(p,v)$ there is an open neighborhood
$(U_{p,v})$ of $p$ foliated by the minimal disks $\{F(p,v,t)(\D)\}_{1-\eps<t<1}=\{D_{p,v,t}\}_{1-\eps<t<1}$.
Let $(p_i,v_i,t_i)\to (\bar p,\bar v,1)$. Let $R_i$ be a minimal
surface in $M$ with $\partial R_i=\Phi(p_i,v_i,t_i)(\S^1)$ with $R_i\neq D_{p_i,v_i,t_i}$.

Let $T_i$ be a minimal disk with boundary $\Phi(p_i,v_i,t_i)(\S^1)$ that
minimizes the area in the connected component of $M\setminus R_i$ that contains
$-p_i$. $T_i$ is an embedded minimal disk and $area(T_i)\le area(\partial
M)$. 

Unless $area (T_i)\to 0$, by Theorem~3 in \cite{Whi3}, $T_i$ converges to a
minimal surface $T$ with $\partial T=\bar p$ and $T$ is smooth outside $\bar p$.
By Theorem~2 in \cite{Whi3}, $T$ is regular even at $\bar p$. So this is
impossible since $\partial M$ is strictly mean convex.

So $area (T_i)\to 0$ and, by the monotonicity formula, $T_i$ can't contain points
outside of $U_{p_i,v_i}$ so the same is true for $R_i$. Since the minimal disks
$\{D_{p_i,v_i,t}\}_{1-\eps<t<1}$ foliate $U_{p_i,v_i}$ this implies that
$R_i=D_{p_i,v_i,t_i}$ by the maximum principle.
\end{remarq}


\subsection{The family of minimal disks}


In this section we are going to prove the following result.

\begin{prop}\label{prop:param}
There is a connected proper $4$-dimensional submanifold $\Sigma\subset \boM$ that contains all
the minimal disks $D_{p,v,t}$ for any $(p,v)$ and $1-\bar\eps\le|t|<1$.
Moreover, for any $\sigma\in \Sigma$, $\delta(\sigma)$ is close to some
$\Phi(p,v,t)$.

Besides there is a map $F:\Sigma\to C^{2,\alpha}(\D,\B)$, continuous in the $C^2$ topology, such that
$\sigma=[F(\sigma)]$.
\end{prop}

Actually $\delta(\sigma)$ can be chosen as close to $\Phi(p,v,t)$ as we want.

First we recall that $\delta:\boM \to C^{2,\alpha}(\S^1,\S^2)$ is proper by Theorem~\ref{th:properness}.


The map $\Phi$ is an embedding of $U\S^2\times (-1,1)$ in
$C^{2,\alpha}(\S^1,\S^2)$. Moreover, by Proposition~\ref{prop:closepole}, when
$1-\bar\eps\le |t|<1$, there is a unique minimal disk $\sigma\in\boM$ such
that $\delta(\sigma)=\Phi(p,v,t)$ and $\sigma$ is non degenerate. So when $t$ is
close to $\pm 1$, $\Phi$ is transverse to the boundary map $\delta$. 

Let us fix $(p_0,v_0)\in U\S^2$. By the Smale transversality theorem \cite{Sma},
there exists $\Phi' : (p_0,v_0)\times(-1,1)\to
C^{2,\alpha}(\S^1,\S^2)$, a perturbation of $\Phi$ on $(p_0,v_0)\times(-1,1)$,
which is transverse to $\delta$. Moreover we can assume that
$\Phi=\Phi'$ when $1-\bar\eps\le|t|<1$.

Thus $\delta^{-1}(\Phi'((p_0,v_0)\times (-1,1)))$ is a submanifold of $\boM$.
Since $\delta$ is proper,
$\delta^{-1}(\Phi'((p_0,v_0)\times (-1,1)))$ is a proper $1$-dimensional
submanifold of $\boM$. Let $\gamma$ be its component that contains
$D_{p_0,v_0,t}$ for $t$ close to $1$. Since $\delta^{-1}(\Phi'((p_0,v_0)\times
(-1,1)))$ is proper, it must contain also
$D_{p_0,v_0,t}$ for $t$ close to $-1$. 

Since $\Phi'$ is a perturbation of $\Phi$ on $(p_0,v_0)\times(-1,1)$, we can
extend its definition to the whole
$U\S^2\times(-1,1)$ as a perturbation of $\Phi$ such that $\Phi'$ is still an
embedding. We notice that a priori $\Phi'$ is not transverse to $\delta$ on
$U\S^2\times(-1,1)$ but it is the case on $(p_0,v_0)\times(-1,1)$ and
$(U\S^2\times ((-1,-1+\bar\eps]\cup[1-\bar\eps,1))$ where we choose
$\Phi'=\Phi$. The Smale transversality theorem implies
there is a perturbation $\wPhi$ of $\Phi'$ on the whole $U\S^2\times(-1,1)$
which is transverse to $\delta$ and coincides with $\Phi'$ on $(p_0,v_0)\times(-1,1)$ and
$(U\S^2\times ((-1,-1+\bar\eps]\cup[1-\bar\eps,1))$.

So $\delta^{-1}(\wPhi(U\S^2\times (-1,1)))$ is a proper $4$-dimensional submanifold of
$\boM$. Inside it, we only consider the connected component $\Sigma$ that
contains $D_{p,v,t}$ for any $(p,v)$ and $t$ close to $\pm1$ ($\Sigma$ exists
since $\gamma\subset\Sigma$).

We denote by $\pi$ the map $\wPhi^{-1}\circ \delta : \Sigma\to
U\S^2\times(-1,1)$. For $t$ close to $\pm 1$ we know that $\pi$ is a bijection.
Let us define
$$
\Sigma_\pm=\delta^{-1}\big(\Phi(U\S^2\times[\pm(1-\bar\eps),\pm1))\big)=\left\{D_{p,v,t};(p,v,t)\in
U\S^2\times[\pm(1-\bar\eps),\pm1)\right\}
$$
Let us also define $\Sigma_0=\delta^{-1}\big(\wPhi(U\S^2\times[-1+\bar\eps,1-\bar\eps])\big)$.

So let us construct the map $F$ that parametrizes all these minimal disks. First
when $\sigma\in\Sigma_\pm$, we know that we can identify $\sigma$ with $(p,v,t)=\pi(\sigma)$ and we define:
$$
F(\sigma)=F(p,v,t): (x,y)\in \D\mapsto \sqrt{1-t^2}(xv+yp\wedge v)+(t+u_{p,v,t}(x,y))p
$$
as in Proposition~\ref{prop:closepole}. So we need to define $F(\sigma)$ when
$\sigma\in\Sigma_0$.

\begin{lem}\label{lem:smale}
There is a map $F:\Sigma\to C^{2,\alpha}(\D,\B)$ continuous in the $C^2$ topology which coincides with
the above definition on $\Sigma_\pm$ such that $\sigma=[F(\sigma)]$. We have
$F(\sigma)_{|\partial\D}=\wPhi(\pi(\sigma))$.
\end{lem}

\begin{proof}
First we notice that $F$ is well defined on $\partial\Sigma_0$. Let us recall
what are the charts around $\bar\sigma\in\boM$. By Theorem~3.3 and Section~8 in \cite{Whi}, if
$\ker D\delta(\bar \sigma)$ has dimension $j$ (the
set of Jacobi fields on $\bar \sigma$ that vanish on the boundary has dimension $j$),
there is a neighborhood of $\bar \sigma$ in $\boM$ that can be identified with a submanifold in
$C^{2,\alpha}(\partial \D,\partial M)\times \R^j$ of codimension $j$ (such that
$\delta$ correspond to the the projection on the first factor). Moreover there
is $H:C^{2,\alpha}(\partial \D,\partial M)\times \R^j\to C^{2,\alpha}(\D,M)$
such that its restriction to the submanifold gives a parametrization of any
minimal surface in a neighborhood of $\bar \sigma$.

Let $Y=Y^0\subset \cdots\subset Y^4$ be a triangulation of $\Sigma_0$ such that
each $4$-cell is contained in one of the above charts of
$\boM$ that are diffeomorphic to a ball. First for any $\sigma$ in $Y^0$, we
choose a parametrization $F(\sigma)$ of $\sigma$ such that, if $\sigma\in
Y^0\cap \partial\Sigma_0$, $F(\sigma)$ coincides with the preceding definition of $F$.

By induction, let us assume that $F:Y^p\to C^{2,\alpha}(\D,\B)$ continuous in
the $C^2$ topology is defined and consider $e$ a $p+1$-cell in $Y^{p+1}$.
If $e$ belongs to $\partial \Sigma_0$, we extend $F$ by the preceding
definition. If not, $e$ belongs to an above chart of $\boM$. So combining the
inclusion of $e$ in the submanifold of $C^{2,\alpha}(\partial\D, \partial M)\times
\R^j$ with $H$, there is a
continuous $X:e\to C^{2,\alpha}(\D,\B)$ such that $[X(\sigma)]=\sigma$. Besides
there is a continuous map $Z:\partial e\to
\diff_{2,\alpha}(\D,\partial\D)$ such that $F(\sigma)=X(\sigma)\circ Z(\sigma)$. Since
$\diff_\infty(\D,\partial\D)$ is contractible (cf. \cite{Sma2} and by
Appendix~\ref{appendix}), $Z$ can be extended
continuously (for the $C^2$ topology) in $\diff_{2,\alpha}(\D,\partial \D)$ to the
whole $e$. We then extend $F$ by $F(\sigma)=X(\sigma)\circ Z(\sigma)\in C^{2,\alpha}(\D,\B)$.
\end{proof}

The proof of Proposition~\ref{prop:param} is finished.


\subsection{Rectifying the parametrization}


From now on, we will forget about the metric $g$ on $\B$, so $\{F(\sigma)\}_{\sigma\in\Sigma}$
is just a family of embedded disks in $\B^3$ with the Euclidean metric.
Actually, the parametrization of
$\sigma$ by $F(\sigma)$ is not good for what we are going to do next. We use the
following result.

\begin{prop}\label{prop:confboundary}
There is a continuous map $Y:\Sigma\to \diff_1(\D,\partial\D)$ such that $F(\sigma)\circ Y(\sigma)$
is
conformal along $\partial \D$. Moreover, if $\sigma\in\Sigma_\pm$, $\lim
Y(\sigma)=\id$ when $t\to \pm 1$.
\end{prop}

\begin{proof}
Actually the proof is based on the following statement: let us consider
$X:\S^1\to \R^2$ a $C^1$ vectorfield such that $\la X,e_r\ra>0$ then there is $Y\in
\diff_1(\D,\partial \D)$ such that $\partial_r Y=X$ on $\partial \D$.

In order to construct $Y$, let $r_0\in(0,1/2)$ and $\phi:[0,1]\to[0,1]$ be a non-increasing function such
that $\phi=1$ on $[0,1-r_0]$, $\phi(1)=0$ and $\phi'(1)=0$. We define
$$
Y(r,\theta)=r\phi e_r+(1-\phi)(e_r+(r-1)X)
$$
We also write $X=\lambda(\cos\alpha e_r+\sin\alpha e_\theta)$ with $\lambda>0$
and $\alpha\in(-\pi/2,\pi/2)$. Thus
\begin{align*}
\partial_rY=&(\phi+r\phi')e_r-\phi'(e_r+(r-1)\lambda(\cos \alpha e_r+\sin\alpha e_\theta))\\
& +(1-\phi)\lambda(\cos \alpha e_r+\sin\alpha e_\theta)\\
\partial_\theta Y=&r\phi e_\theta+(1-\phi)(e_\theta +(r-1)(\lambda_\theta
\cos\alpha e_r+\lambda_\theta \sin\alpha e_\theta\\
&-\lambda\alpha_\theta \sin\alpha e_r+\lambda\alpha_\theta \cos\alpha
e_\theta+\lambda\cos\alpha e_\theta-\lambda\sin\alpha e_r))\\
\end{align*}
Choosing $r_0$ such that $2r_0\max(\lambda,\lambda_\theta,\lambda\alpha_\theta)$
is small and using $\phi'=0$ on $[0,1-r_0]$,
\begin{align*}
\partial_rY=&(\phi+r\phi'-\phi'(1+(r-1)\lambda\cos\alpha)+(1-\phi)\lambda\cos\alpha)e_r\\
&+(-\phi'(r-1)\lambda\sin\alpha+(1-\phi)\lambda\sin\alpha)e_\theta\\
\partial_\theta Y=&(r\phi+1-\phi) e_\theta+\eps(r,\theta)
\end{align*}
where $\eps(r,\theta)$ can be assumed small by reducing $r_0$ and is vanishing
if $r<1-r_0$. So the Euclidean Jacobian is
$$
J=\frac1r(\phi+r\phi'-\phi'(1+(r-1)\lambda\cos\alpha)+(1-\phi)\lambda\cos\alpha)(r\phi+1-\phi)+\frac1r\eps
$$
since $\phi$ will be chosen such that $r_0\phi'$ is bounded.

If $r\le 1-r_0$, we get $J=1$. So we focus on the sign of 
\begin{multline*}
(\phi+r\phi'-\phi'(1+(r-1)\lambda\cos\alpha)+(1-\phi)\lambda\cos\alpha)=\\
((\phi+(r-1)\phi')+\lambda\cos\alpha(1-(\phi+(r-1)\phi'))
\end{multline*}
We will choose $\phi$ such that this quantity is positive. We can assume that
$\lambda\le \lambda_0$, $\lambda\cos\alpha\ge \eta$ where $\lambda_0$ and
$\eta\in(0,1)$ are some constant. Since
$\phi'\le 0$, we have $\phi+(r-1)\phi'\ge \phi\ge 0$. So if
$\phi+(r-1)\phi'\le 1$ and $\lambda\cos\alpha\ge 1$, one has
\begin{align*}
\phi+(r-1)\phi'+\lambda\cos\alpha(1-\phi-(r-1)\phi')&\ge
\lambda\cos\alpha+(\phi+(r-1)\phi')(1-\lambda\cos\alpha)\\
&\ge \lambda\cos\alpha+(1-\lambda\cos\alpha)\ge 1
\end{align*}
If $\phi+(r-1)\phi'\le 1$ and $\lambda\cos\alpha\le 1$,
\begin{align*}
\phi+(r-1)\phi'+\lambda\cos\alpha(1-\phi-(r-1)\phi')&\ge
\lambda\cos\alpha+(\phi+(r-1)\phi')(1-\lambda\cos\alpha)\\
&\ge \lambda\cos\alpha\ge \eta
\end{align*}
If $\phi+(r-1)\phi'\ge 1$ we have
\begin{align*}
\phi+(r-1)\phi'+\lambda\cos\alpha(1-\phi-(r-1)\phi')&\ge
\phi+(r-1)\phi'+\lambda_0(1-(\phi+(r-1)\phi'))\\
&\ge (r\phi+(1-\phi)(1+\lambda_0(r-1)))'
\end{align*}
So we will choose $\phi$ such that $(r\phi+(1-\phi)(1+\lambda_0(r-1)))'>0$. Actually
we choose $\psi:[0,1]\to[0,1]$ non increasing with $\psi(1)=0$, $\psi'(1)=0$ and $\psi(r)=1$
on $[0,1/2]$ and such that $(r\psi+(1-\psi)(1+\lambda_0(r-1)))'>0$ on $[1/2,1]$.
Then $\phi$ is defined by
$$
\phi(r)=
\begin{cases}
1&\text{if }r\le 1-r_0\\
\psi(1+\frac{r-1}{r_0})&\text{if }r\ge 1-r_0\\
\end{cases}
$$
Thus the expected estimates about $\phi$ come from the ones on $\psi$ and are
independent of $r_0$. Thus $Y$ is a $C^1$ local diffeomorphism and then an open
map. Since
$Y_{|\partial \D}=\id$, $Y(\D)\subset \D$ and $Y$ is a global diffeomorphism.
Moreover $\partial_rY=X$.

So in order to conclude the proof of Proposition~\ref{prop:confboundary}, we
will apply the above construction to the vector fields $X(\sigma)$ given for
$\zeta\in\partial \D$ by 
$$
X(\sigma)(\zeta)=D(F(\sigma))^{-1}(\zeta)\big(r_{-\pi/2}(\partial_\theta F(\sigma)(\zeta))\big)
$$
where $r_{-\pi/2}$ is the rotation by angle $-\pi/2$ in $T_{F(\sigma)(\zeta)}\sigma$.

Let us remark that when $\pi(\sigma)=(p,v,t)$ with $t$ close to $1$, we have
$$
F(\sigma)(x,y)=\sqrt{1-t^2}(xv+yp\wedge v)+(t+u_{p,v,t}(x,y))p
$$
So $\partial_\theta F(\sigma)(e^{i\theta})=\sqrt{1-t^2}(-\sin\theta v+\cos\theta
p\wedge v)+O(1-t)$ and $\partial_rF(\sigma)(e^{i\theta})=\sqrt{1-t^2}(\cos\theta
v+\sin\theta p\wedge v)+O(1-t)$. So $r_{-\pi/2}(\partial_\theta
F(\sigma)(e^{i\theta}))=\sqrt{1-t^2}(\cos\theta v+\sin\theta p\wedge
v)+O(1-t)=\partial_rF(\sigma)(e^{i\theta})+O(1-t)$. This give us
$X(\sigma)(e^{i\theta})=e_r+O(\sqrt{1-t})$. The same is true for $t$ near $-1$.

As a consequence, all the estimates that appear in the construction can be
chosen uniformly in $\sigma$. So $r_0$ can be chosen independently of $\sigma$
and $Y$ depends continuously on $\sigma$. The last remark is that, as
$t\to\pm1$, $Y(\sigma)\to \id$.
\end{proof}

In the sequel, we denote $\widetilde F(\sigma)=F(\sigma)\circ Y(\sigma)$ which
is conformal on the boundary. 


\subsection{Extending $\widetilde F$ and the boundary behaviour}


We denote by $\nu(\sigma)=\frac{\partial_x \widetilde F(\sigma)\wedge \partial_y
\widetilde F(\sigma)}{\|\partial_x \widetilde F(\sigma)\wedge \partial_y
\widetilde F(\sigma)\|}$ the Euclidean unit normal to $\sigma$. We also define
$h(\sigma)=\frac{\partial_x \widetilde F(\sigma)}{\|\partial_x \widetilde
F(\sigma)\|}$. Finally we define $H:\Sigma\times \D\to \B\times U\S^2$ by
$$
H(\sigma,\zeta)=(\widetilde F(\sigma)(\zeta),\nu(\sigma)(\zeta),h(\sigma)(\zeta)).
$$

Looking at the parametrization $\widetilde F(\sigma)$ when $\sigma\in
\Sigma_{\pm}$, we have $H(\sigma,\zeta)\to (\pm p,p,v)$ when $\pi(\sigma)\to
(p,v,\pm 1)$. This allows us to compactify $\Sigma\times\D$ and extend $H$ to
this compactification.

Since $(\pi,\id)$ is a bijection from $\Sigma_+\times \D\to
U\S^2\times[1-\bar\eps,1)\times \D$, one can extend $\Sigma_+\times \D$ as
$U\S^2\times[1-\bar\eps,1]\times \D$ then take the quotient by the relation
$(p,v,1,\zeta)\sim(p,v,1,\zeta')$. Its like compactifying $[1-\bar\eps,1)\times \D$ as the
upper part of a $3$-ball. Then $H$ extends by continuity by $H(p,v,1,\zeta)=(p,p,v)$. The
same can be done for $\Sigma_-$.

So we get a compact manifold $K$ where $H$ extends to $G:K\to \B\times U\S^2$.
We need to describe $G$ on $\partial K$:
\begin{itemize}
\item if $k\in \partial K$ is a point added to $\Sigma\times \D$ along the
compactification corresponding to $(p,v,1 )$, we have $G(k)=(p,p,v)$;
\item if $k\in \partial K$ is a point added to $\Sigma\times \D$ along the
compactification corresponding to $(p,v,-1)$, we have $G(k)=(-p,p,v)$;
\item if $k$ is not an added point and $k=(\sigma,\zeta)$ with $|\zeta|=1$, we
have $G(k)=(\delta(\sigma)(\zeta),\nu(\sigma)(\zeta),h(\sigma)(\zeta))$. If
$\pi(\sigma)=(p,v,t)$, $\partial(\sigma)(\zeta)=\wPhi(p,v,t)(\zeta)$ which is
close to $\Phi(p,v,t)(\zeta)$ (and even equal on $\Sigma_\pm\times \partial
\D$). $\nu(\sigma)(\zeta)$ is a unit normal vector to $\wPhi(p,v,t)$. Since
$\widetilde F(\sigma)$ is conformal along $\partial \D$, we have:
\begin{align*}
h(\sigma)(e^{i\theta})&=\frac{\cos\theta \partial_r\widetilde
F(\sigma)-\sin\theta\partial_\theta\widetilde
F(\sigma)}{\|\partial_\theta\widetilde F(\sigma)\|} (e^{i\theta})\\
&=\frac{\cos\theta \partial_r\widetilde
F(\sigma)-\sin\theta\partial_\theta\wPhi(\pi(\sigma))}{\|\partial_\theta\wPhi(\pi(\sigma))\|}
(e^{i\theta})\\
&=\cos\theta\left(\frac{\partial_\theta\wPhi(\pi(\sigma))}
{\|\partial_\theta\wPhi(\pi(\sigma))\|}\wedge \nu\right)-\sin\theta
\frac{\partial_\theta\wPhi(\pi(\sigma))} {\|\partial_\theta\wPhi(\pi(\sigma))\|}\\
&=\cos\theta\tilde e_\theta\wedge \nu-\sin\theta\tilde e_\theta
\end{align*}
where $\tilde e_\theta=\frac{\partial_\theta\wPhi(\pi(\sigma))} {\|\partial_\theta\wPhi(\pi(\sigma))\|}$.
\end{itemize} 

Actually we have
\begin{prop}\label{prop:homotopy}
On $\partial K$, choosing $\wPhi$ close enough to $\Phi$, $G$ is homotopic to
$\widetilde G:\partial K\to \S^2\times U\S^2;(\sigma,\zeta)\mapsto
\big(\Phi(\pi(\sigma))(\zeta),\pi_1(\sigma)\big)$ where $\pi_1(\sigma)=(p,v)$ if
$\pi(\sigma)=(p,v,t)$.
\end{prop}

\begin{proof}
In order to construct this homotopy we can construct it for each factor $\S^2$
and $U\S^2$. We focus on the second one.

To do this, we endow $U\S^2$ with a Riemannian metric such that the following identifications
$\R P^3\simeq SO_3\simeq U\S^2$ are isometries. Let us recall that these
identifications are constructed in the following way. If $(a,b,c,d)\in \S^3$, the
unit norm quaternion $q=a+bi+cj+dk$ acts by conjugation on the unit sphere of
purely imaginary quaternions by the following matrix of $SO_3$:
$$
\begin{pmatrix}
a^2+b^2-c^2-d^2&2(-ad+bc)&2(ac+bd)\\
2(ad+bc)&a^2-b^2+c^2-d^2&2(-ab+cd)\\
2(-ac+bd)&2(ab+cd)&(a^2-b^2-c^2+d^2)\\
\end{pmatrix}
$$
We notice that only $(a,b,c,d)$ and $(-a,-b,-c,-d)$ are identified with the
above matrix; so $\R P^3\simeq SO_3$. Besides, if $(p,v)\in U\S^2$, we can
consider the matrix $(p,v,p\wedge v)\in SO_3$. This gives the second
identification $U\S^2\simeq SO_3$. Finally we use them to put on $U\S^2$ the
Riemannian metric of $\R P^3$ inherited from the one on $\S^3$. 

As a consequence, the cut-locus of $((1,0,0),(0,1,0))$ (which is
$[a:b:c:d]=[1:0:0:0]\in \R P^3$) is given by $\{a=0\}$. So the cut-locus of
$I_3$ is the set of matrices of trace $-1$: the rotation of angle $\pi$.

In order to construct the homotopy, we are going to prove that for any
$(\sigma,\zeta)$ the matrix $\big(\nu(\sigma)(\zeta),h(\sigma)(\zeta),
\nu(\sigma)(\zeta)\wedge h(\sigma)(\zeta)\big)$ is not in the cut-locus of
$(p,v,p\wedge v)$ where $\pi_1(\sigma)=(p,v)$. Because of the group structure of
$SO_3$ and the invariance of the below arguments by left multiplication, we assume
that $\pi_1(\sigma)=(p_0,v_0)=((1,0,0),(0,1,0))$. Moreover we recall that $\Phi$
is close to $\wPhi$ and even equal when $t$ is close to $\pm1$.

At $\wPhi(p_0,v_0,t)(e^{i\theta})$, the unit vector $\tilde e_\theta$ is close
to $e_\theta=(0,-\sin\theta,\cos\theta)$. So we can write $\tilde
e_\theta=(\sin\alpha,-\cos\alpha\sin\theta',\cos\alpha\cos\theta')$ with
$\alpha$ close to $0$ and $\theta'$ close to $\theta$. We write
$\nu=(\cos\phi,\sin\phi\cos\beta,\sin\phi\sin\beta)$ with $\beta$ close to
$\theta$ and even equal if $t$ close to $\pm 1$. Thus
$$
\tilde e_\theta\wedge\nu=\begin{pmatrix}
\sin\alpha\\ -\cos\alpha\sin\theta'\\\cos\alpha\cos\theta'
\end{pmatrix}
\wedge \begin{pmatrix}
\cos\phi\\\sin\phi\cos\beta\\ \sin\phi\sin\beta
\end{pmatrix}=
\begin{pmatrix}
-\cos\alpha\sin\phi \cos(\theta'-\beta)\\
-\sin\alpha\sin\phi\sin\beta+\cos\alpha\cos\phi\cos\theta'\\
\sin\alpha \sin\phi \cos\beta+\cos\alpha\cos\phi\sin\theta'
\end{pmatrix}
$$
Moreover $\nu\wedge h=\nu\wedge(\cos\theta \tilde e_\theta\wedge \nu-\sin\theta
\tilde e_\theta)=\cos\theta \tilde e_\theta+\sin\theta \tilde
e_\theta\wedge\nu$. Thus the trace of $(\nu,h,\nu\wedge h)$ takes the value
\begin{align*}
\tr(\nu,h,\nu\wedge
h)&=\cos\phi+\cos\theta(-\sin\alpha\sin\phi\sin\beta+\cos\alpha\cos\phi\cos\theta')\\
&\qquad+\sin\theta\cos\alpha\sin\theta'+\cos\theta\cos\alpha\cos\theta'\\
&\qquad+\sin\theta(\sin\alpha \sin\phi \cos\beta+\cos\alpha\cos\phi\sin\theta')\\
&=\cos\phi+\sin\alpha\sin\phi\sin(\theta-\beta)+\cos\alpha\cos\phi\cos(\theta-\theta')\\
&\qquad +\cos\alpha\cos(\theta-\theta')
\end{align*}
As $\alpha$ is close to $0$ and $\theta'$ close to $\theta$, the trace is close
to $1+2\cos\phi$.

For $|t|\ge 1-\bar\eps$, we know $\theta=\theta'=\beta$ and $\alpha=0$. Moreover
$\phi$ is close to $0$ (its value when $t=\pm1$). So the value of the trace is
close to $3$ for $|t|\ge \cos\xi$ (for some positive $\xi$ close to $0$).

For $t$ not close to $\pm1$, we know that $\tilde e_\theta\wedge \nu$ points to
the outside of $\B$ at
$\wPhi(p_0,v_0,t)(e^{i\theta})=(\cos\gamma,\sin\gamma\cos\theta'',\sin\gamma\sin\theta'')$
where $\cos\gamma$ is close to $t$ and $\theta''$ is close to $\theta$ since
$\nu$ is never normal to $\S^2=\partial\B$. So
\begin{align*}
0<\la \tilde e_\theta&\wedge \nu,\wPhi(p_0,v_0,t)(e^{i\theta})\ra\\
&=-\cos\gamma\cos\alpha\sin\phi
\cos(\theta'-\beta) \\
&\qquad +\sin\gamma\cos\theta''(-\sin\alpha\sin\phi\sin\beta+\cos\alpha\cos\phi\cos\theta')\\
&\qquad +\sin\gamma\sin\theta''(\sin\alpha \sin\phi \cos\beta+\cos\alpha\cos\phi\sin\theta')\\
&=-\cos\gamma\cos\alpha\sin\phi \cos(\theta'-\beta)+\sin\gamma\sin\alpha\sin\phi\sin(\theta''-\beta)\\
&\qquad +\sin\gamma\cos\alpha\cos\phi\cos(\theta''-\theta')
\end{align*}
As $\theta\simeq\theta'\simeq\theta''\simeq \beta$ and $\alpha\simeq 0$, we get
$0<\sin(\gamma-\phi)+\eps$. So $\gamma-\pi-\eta<\phi<\gamma+\pi+\eta$ for some
small $\eta$. We notice that all the above approximations depends on how far
$\wPhi$ is from $\Phi$. Since we can choose $\wPhi$ as close of $\Phi$ as we
want, we can make all these approximations very precise. So where we have to
pertub $\Phi$ into $\wPhi$, $\xi\le \gamma\le \pi-\xi$. Since we can assume
$\eta<\xi$, we get 
$$
-\pi<-\pi+\xi-\eta<\phi<\pi+\eta-\xi<\pi
$$
So $\tr(\nu,h,\nu\wedge h)\simeq 1+2\cos\phi>-1$. Thus choosing $\wPhi$ close
enough to $\Phi$, we can be sure that, for any $(\sigma,\zeta)$,
$(\nu(\sigma)(\zeta),h(\sigma)(\zeta))$ is not in the cut-locus of
$\pi_1(\sigma)=(p,v)$. Thus moving the points along the geodesic from
$(\nu(\sigma)(\zeta),h(\sigma)(\zeta))$ to $\pi_1(\sigma)$, we get an homotopy
from $(\sigma,\zeta)\mapsto (\nu(\sigma)(\zeta),h(\sigma)(\zeta))$ to
$(\sigma,\zeta)\mapsto \pi_1(\sigma)$. This finishes the proof for the second
factor.

For the first one, we only remark that $(\sigma,\zeta)\mapsto
\wPhi(\pi(\sigma)(\zeta))$ and $(\sigma,\zeta)\mapsto
\Phi(\pi(\sigma)(\zeta))$ are close since $\wPhi$ is a perturbation of $\Phi$
(and we have equality for points added in the compactification). So the two
 factors are homotopic. This finishes the proof of Proposition~\ref{prop:homotopy}.
\end{proof}


\subsection{The final argument}


We are going to prove that $G$ is surjective on the interior of $\B\times
U\S^2$. If it is true, this will tell us that for any $(q,\nu)\in \B\times \S^2$
there is a disk $\sigma$ passing through $q$ with unit normal $\nu$: this is
exactly the statement of Theorem~\ref{th:prescrip}.

Let us assume that this is not the case. Let $A=(a,b)\in \B\times U\S^2$ be not in
the image of $G$. Let us notice that $\B\setminus\{a\}$ can be deformation retracted to
$\S^2$ and $U\S^2\setminus\{b\}$ can be deformation retracted to the equatorial $\R P^2$ in
$\R P^3\simeq U\S^2$ with pole at $b$. So $\B\times U\S^2\setminus \{A\}$ can be deformation
retracted to $\Delta=(\S^2\times U\S^2)\cup_{\S^2\times \R P^2}(\B\times \R P^2)$
where both terms are glued together along the common $\S^2\times \R P^2$.

Let us consider homology with $\Z/2\Z$ coefficients since it is not clear if $K$
is orientable. Composing $G$ with the deformation retract defines a map $G':K\to
\Delta$ which coincides with $G$ on $\partial K$. As a consequence, $[G(\partial
K)]=[G'(\partial K)]=[\partial G'(K)]=0\in H_5(\Delta,\Z/2\Z)$.

A part of the Mayer-Vietoris sequence associated to $\Delta=(\S^2\times
U\S^2)\cup_{\S^2\times \R P^2}(\B\times \R P^2)$ is 
$$
H_5(\S^2\times \R P^2)\to H_5(\S^2\times U\S^2)\oplus H_5(\B\times \R P^2)\to H_5(\Delta)
$$
Since $H_5(\S^2\times \R P^2)=0$ and $H_5(\B\times\R P^2)=0$, the
inclusion $i:\S^2\times U\S^2\hookrightarrow \Delta$ gives an injective inclusion
$H_5(\S^2\times U\S^2,\Z/2\Z)\to H_5(Z,\Z/2\Z)$. $G$ is homotopic to $\widetilde
G :\partial K\to \S^2\times
U\S^2\subset Z; (\sigma,\zeta)\mapsto (\Phi(\pi(\sigma))(\zeta),\pi_1(\sigma))$
which is a degree $1$ map: indeed when $t$ is close to $1$ there is exactly one
antecedent. So $[G(\partial K )]= [\widetilde G(\partial K)]\neq 0\in
H_5(Z,\Z/2\Z)$. Thus we have a contradiction and $G$ is surjective.


\section{A minimal disk containing three points}
\label{sec:3point}


In this section, we do similar arguments to the preceding section in order to
prove that choosing three points in a Riemannian ball there is a minimal disk
containing these three points.

\begin{thm}\label{th:3point}
Let $M=(\B,g)$ be as in Theorem~\ref{th:prescrip}. Let $q_1,q_2,q_3\in M$ be
three points, then there is
$\sigma\in\boM$ such that $q_i\in \sigma$ for 1$\le i\le 3$. 
\end{thm}

In order to do the proof we need an equivariant version of the construction of
the preceding section. Let $R$ and $S$ be defined as maps
$C^{2,\alpha}(\S^1,\S^2)\to C^{2,\alpha}(\S^1,\S^2)$ or $C^{2,\alpha}(\D,\B)\to
C^{2,\alpha}(\D,\B)$ by $R(X)(z)=X(-z)$ and $S(X)(z)=X(\bar z)$. We notice that
these maps induce maps $R$ and $S$ on $\boM$ such that the boundary map $\delta$
is equivariant. We also notice that $R$ and $S$ generate a free actions of $G
=(\Z/2\Z)^2$ on $\boM$ and $C^{2,\alpha}(\S^1,\S^2)$. 

On $U\S^2\times (-1,1)$ we also define $r(p,v,t)=(p,-v,t)$ and
$s(p,v,t)=(-p,v,-t)$. It also generates a free action of $G=(\Z/2\Z)^2$ on
$U\S^2\times (-1,1)$ such that $\Phi\circ r = R\circ \Phi$ and $\Phi\circ
s=S\circ \Phi$ where $\Phi$ is defined by~\eqref{eq:Phi}. 

Let us denote $\widetilde \boM=\boM/G$, $\widetilde U=(U\S^2\times (-1,1))/G$,
$\widetilde \Gamma=C^{2,\alpha}(\S^1,\S^2)/G$; we denote by $\Pi:\boM\to
\widetilde\boM$ the projection map. Let $\Psi$ and $\tilde\delta$ be
the induced map from $\Phi$ and $\delta$. We then have a commutative diagram
\begin{center}
\begin{tikzcd}
&\widetilde\boM\arrow[d,"\tilde\delta"]\\
\widetilde U\arrow[r,hook,"\Psi"']&\widetilde\Gamma
\end{tikzcd}
\end{center}

Since the action of $G$ on $\boM$ and $C^{2,\alpha}(\S^1,\S^2)$ are free and $G$
is finite, $\widetilde\boM$ and $\widetilde\Gamma$ are Banach manifolds and
$\tilde\delta$ is Fredholm of index $0$. Moreover $\widetilde U$ is a smooth
manifold and $\Psi$ is an embedding. So now the idea is to do an equivariant
version of the work in Section~\ref{sec:prescrib}.


\subsection{The equivariant construction}


As in the preceding section (Proposition~\ref{prop:closepole}), when $t$ is
close to $\pm1$, there is a unique
$\sigma\in \widetilde M$ such that $\tilde\delta(\sigma)=\Psi(p,v,t)$ and
$D\tilde\delta$ is invertible at $\sigma$. So we can perturb $\Psi$ into
$\widetilde\Psi:\widetilde U\to \widetilde \Gamma$ such that $\tilde\delta$ and
$\widetilde\Psi$ are transverse and $\Psi=\widetilde\Psi$ close to $t=\pm1$.

As above $\widetilde\Sigma=\tilde\delta^{-1}(\widetilde\Psi(\widetilde U))$ is
a proper smooth submanifold of $\widetilde\boM$. Let
$\Sigma=\Pi^{-1}(\widetilde\Sigma)\subset \boM$.

As above, we want to have a parametrization of all the
minimal disks in $\Sigma$.
 
When $t$ is close to $\pm 1$, we have the parametrizations
$$
F(p,v,t):(x,y)\mapsto\sqrt{1-t^2}(xv+yp\wedge v)+(t+u_{p,v,t}(x,y))p
$$
We notice that, by the implicit function theorem, we have
$$
F(r(p,v,t))(z)=F(p,v,t)(-z) \text{ and }F(s(p,v,t))(z)=F(p,v,t)(\bar z)
$$
As above we define $\Sigma_\pm=\left\{D_{p,v,t};(p,v,t)\in
U\S^2\times[\pm(1-\bar\eps),\pm1)\right\}$ and $\Sigma_0=\barre{\Sigma\setminus(\cup
\Sigma_\pm)}$.
So we want to extend this to the whole $\Sigma$ in an equivariant way.
 
\begin{lem}
There is a continuous map $F:\Sigma\to C^{2,\alpha}(\D,\B)$ in the $C^2$
topology which coincides with the above
definition on $\Sigma_\pm$ such that $[F(\sigma)]=\sigma$ and $F$ is
equivariant:
$$
F(R(\sigma))(z)=F(\sigma)(-z) \text{ and }F(S(\sigma))(z)=F(\sigma)(\bar z)
$$
\end{lem}
\begin{proof}
Let $\widetilde\Sigma_0$ be $\Pi(\Sigma_0)$. Let $\widetilde Y^0\subset
\cdots\subset \widetilde Y^4$ be a triangulation of
$\widetilde\Sigma_0$ such each $4$-cell lifts to $\Sigma_0$ as $4$
disjoint $4$-cells. We denote by $Y^0\subset \cdots\subset Y^4$ the lift of
this triangulation. So if $e$ is a $p$-cell in $Y^p$, $R(e)$, $S(e)$ and
$R\circ S(e)$ are the other $p$-cells above $\Pi(e)$.
 
As above the proof is by induction. For any $0$-cell $\tilde e$ in $\widetilde
Y^0$, let $\{e,R(e),S(e),R\circ S(e)\}$ be $\Pi^{-1}(\tilde e)$. Let $F(e)\in
C^{2,\alpha}(\D,\B)$ be a parametrization of $e$ and define respectively the
 parametrizations of $R(e),S(e),R\circ S(e)$ by $F(R(e))(z)=F(e)(-z)$,
 $F(S(e))(z)=F(e)(\bar z)$ and $F(R\circ S(e))(z)=F(e)(-\bar z)$. Since $e$, $R(e)$,
 $S(e)$ and $R\circ S(e)$ are disjoint this is well defined and moreover $F$ is
 $G$-equivariant.
 
 By induction let us assume that an equivariant $F:Y^p\to C^{2,\alpha}(\D,\B)$
 is defined. Consider $\tilde e$ a $p+1$ cell in $\widetilde Y^{p+1}$ and
 $\{e,R(e),S(e),R\circ S(e)\}$ be $\Pi^{-1}(\tilde e)$. As in Lemma~\ref{lem:smale}, $F$ is defined
 on $\partial e$ and can be extended to the interior of $e$. Then we define, for
 any $\sigma\in e$, $F(R(\sigma))(z)=F(\sigma)(-z)$, $F(S(\sigma))(z)=F(\sigma)(\bar
 z)$ and $F(R\circ S(\sigma))(z)=F(\sigma)(-\bar z)$. This extends the definition
 of $F$ to $S(e)$, $R(e)$ and $R\circ S(e)$. So $F$ is well defined on
 $\Pi^{-1}(\tilde e)$ in an equivariant way. We then end the construction by induction.
 \end{proof}
 
\subsection{The degree argument}


Let us now define $H:\Sigma\times [-1,1]\times \D\times \D \to \B\times \B \times \B$ by 
$H(\sigma,x_1,z_2,z_3)=(F(\sigma)(x_1),F(\sigma)(z_2),F(\sigma)(z_3))$. Our
goal is to prove that $H$ is surjective. We notice that $x_1$ is real.
 
On $\Sigma\times [-1,1]\times \D^2$, there is an action of $G=(\Z/2\Z)^2$ which is defined by 
$$
\barre R(\sigma,x_1,z_2,z_3)=(R(\sigma),-x_1,-z_2,-z_3) \text{ and } \barre
S(\sigma,x_1,z_2,z_3)=(S(\sigma),x_1,\bar z_2,\bar z_3)
$$
Because of the equivariance of $F$, we have $H\circ \barre R=H$ and $H\circ
\barre S=H$. As a consequence the map $H$ induces a map $\widetilde H$ on the
quotient 
$\widetilde K=\Sigma\times [-1,1]\times \D^2/G$. It is enough to prove that
$\widetilde H$ is surjective. Let $\widetilde L$ be the interior of $\widetilde
K$ \textit{i.e.} $(\sigma,x_1,z_2,z_3)\in \widetilde L$ if
$|x_1|<1$, $|z_2|<1$ and $|z_3|<1$.

\begin{lem}
The map $\widetilde H: \widetilde L\to \inter \B\times \inter \B\times \inter
\B$ has mod 2 degree equal to $1$.
\end{lem}
\begin{proof}
Actually we are going to prove that $\widetilde H$ has $\Z/2\Z$ degree one. It
is clear that $\widetilde H: \widetilde
L\to \inter \B\times \inter \B\times \inter \B$ is proper. So it has a well
defined mod 2 degree.

Let us compute this degree. Let $(p,v)\in U\S^2$ and $t$ close to $1$. We are
interested by $H^{-1}(tp+\sqrt{1-t^2}v,tp+\sqrt{1-t^2}p\wedge
v,tp-\sqrt{1-t^2}v)$. Clearly it is made of four points in $K$:
\begin{multline*}
(F(p,v,t),1,i,-1),(F(p,-v,t),-1,-i,1),(F(-p,v,-t),1,-i,-1),\\
(F(-p,-v,-t),-1,-i,1)
\end{multline*}
We recall that $x_1$ is real here. As a consequence, $\widetilde
H^{-1}(tp+\sqrt{1-t^2}v,tp+\sqrt{1-t^2}p\wedge v,tp-\sqrt{1-t^2}v)$ has only
one element. Let us see that it is also the case of $\widetilde
H^{-1}(q_1,q_2,q_3)$ for $(q_1,q_2,q_3)\in \inter \B \times \inter \B \times
\inter \B$ close to $(tp+\sqrt{1-t^2}v,tp+\sqrt{1-t^2}p\wedge
v,tp-\sqrt{1-t^2}v)$. First we remark that, if
\begin{multline*}
(F(p_n,v_n,t_n)(x_{1,n}),F(p_n,v_n,t_n)(z_{2,n}),F(p_n,v_n,t_n)(z_{3,n}))
\longrightarrow\\
(tp+\sqrt{1-t^2}v,tp+\sqrt{1-t^2}p\wedge v,tp-\sqrt{1-t^2}v),
\end{multline*}
then
$(p_n,v_n,t_n,x_{1,n},z_{2,n},z_{3,n})\to (p,v,t,1,i,-1)$. Indeed, we can assume
that $(p_n,v_n,t_n,x_{1,n},z_{2,n},z_{3,n})\to (\bar p, \bar v, \bar
t,x_1,z_2,z_3)$ and we notice that $\bar t\neq \pm 1$ since otherwise
$(F(p_n,v_n,t_n)(x_{1,n}),F(p_n,v_n,t_n)(z_{2,n}),F(p_n,v_n,t_n)(z_{3,n}))\to
\pm (\bar p,\bar p,\bar p)$. Once $t_n\to \bar t\in (-1,1)$, the convergence to
$(p,v,t,1,i,-1)$ is clear. So we can focus on a neighborhood of
$(p,v,t,1,i,-1)$.

Let us compute the differential of $H(p,v,t,x_1,x_2,y_2,x_3,y_3)$ at
$(p,v,t,1,0,1,-1,0)$. Actually we consider the differential of
$$
h(p,v,\theta,x_1,x_2,y_2,x_3,y_3)=H(p,v,\cos\theta,x_1,x_2,y_2,x_3,y_3).
$$
We
notice that the tangent space to $(p,v)$ in $U\S^2$ is $\{(q,w)\in T_p\S^2\times
T_v\S^2|(q,v)+(p,w)=0\}$. So
$$
D_{(p,v)}h(q,w)=(\cos\theta q+\sin\theta w,\cos\theta q+\sin\theta q\wedge
v+\sin\theta p\wedge w,\cos\theta q-\sin\theta w)
$$
We write $(q,w)=(av+bp\wedge v,-ap+cp\wedge v)$ so $q\wedge v=-b
p$ and $p\wedge w=-cv$ and this derivative becomes
\begin{align*}
D_{(p,v)}h(q,w)=(&(-\sin\theta p+\cos\theta v)a +\cos\theta bp\wedge v+\sin\theta cp\wedge v,\\
&\cos\theta av+(\cos\theta p\wedge v-\sin\theta p)b-\sin\theta cv,\\
&(\sin\theta p+\cos\theta v)a+\cos\theta bp\wedge v-\sin\theta cp\wedge v)
\end{align*}
For the other derivatives we have
\begin{align*}
\partial_\theta h&=(-\sin\theta p+\cos\theta v,-\sin\theta v+\cos\theta p\wedge
v,-\sin\theta p-\cos\theta v)\\
\partial_{x_1} h&=(\partial_x u_{p,v,\sin\theta}(1,0)p+\sin\theta v,0,0)\\
\partial_{x_2} h&=(0,\partial_x u_{p,v,\sin\theta}(0,1)p+\sin\theta v,0)=(0,\sin\theta v,0)\\
\partial_{y_2} h&=(0,\partial_y u_{p,v,\sin\theta}(0,1) p+\sin\theta p\wedge v,0)\\
\partial_{x_3} h&=(0,0,\partial_x u_{p,v,\sin\theta}(-1,0)p+\sin\theta v)\\
\partial_{y_3} h&=(0,0,\partial_y u_{p,v,\sin\theta}(-1,0) p+\sin\theta p\wedge
v)=(0,0,\sin\theta p\wedge v)\\
\end{align*}

We notice that, by Proposition~\ref{prop:closepole}, $\nabla
u_{p,v,\cos\theta}=O(\sin^2\theta)$. So considering the
family $(D_{p,v}h(v,-p),D_{p,v}h(p\wedge
v,0),D_{p,0}h(0,p\wedge v),\partial_\theta h,\partial_{x_1} h,\partial_{x_2}
h,\partial_{y_2} h,\partial_{x_3} h,\partial_{y_3} h)$ in the basis
$(p,0,0),(p\wedge v,0,0),(0,p,0),(0,0,p),(v,0,0)(0,v,0),(0,p\wedge v,0),(0,0,v),
(0,0,p\wedge v)$ the jacobian matrix is 
$$
\left(\begin{smallmatrix}
-\sin\theta&0&0&-\sin\theta&\partial_xu_{p,v,\cos\theta}(1,0)&0&0&0&0\\
0&\cos\theta&\sin\theta&0&0&0&0&0&0\\
0&-\sin\theta&0&0&0&\partial_yu_{p,v,\cos\theta}(0,1)&0&0&0\\
\sin\theta&0&0&-\sin\theta&0&0&0&\partial_x u_{p,v,\cos\theta}(-1,0)&0\\
\cos\theta&0&0&\cos\theta&\sin\theta&0&0&0&0\\
\cos\theta&0&-\sin\theta&-\sin\theta&0&\sin\theta&0&0&0\\
0&-\cos\theta&0&\cos\theta&0&0&\sin\theta&0&0\\
\cos\theta&0&0&-\cos\theta&0&0&0&\sin\theta&0\\
0&\cos\theta&-\sin\theta&0&0&0&0&0&\sin\theta
\end{smallmatrix}\right)
$$
So this matrix has the form

$$
\sin\theta
\begin{pmatrix}
-1&0&0&-1&0&0&0&0&0\\
0&\cot\theta&1&0&0&0&0&0&0\\
0&-1&0&0&0&0&0&0&0\\
1&0&0&-1&0&0&0&0&0\\
\cot\theta&0&0&\cot\theta&1&0&0&0&0\\
\cot\theta&0&-1&-1&0&1&0&0&0\\
0&-\cot\theta&0&\cot\theta&0&0&1&0&0\\
\cot\theta&0&0&-\cot\theta&0&0&0&1&0\\
0&\cot\theta&-1&0&0&0&0&0&1
\end{pmatrix}+O(\sin^2\theta)
$$
So if $\theta$ is sufficiently close to $0$ its inverse is
$$
\frac1{\sin\theta}\bigg(
\begin{pmatrix}
-\frac12&0&0&-\frac12&0&0&0&0&0\\
0&0&-1&0&0&0&0&0&0\\
0&1&c&0&0&0&0&0&0\\
-\frac12&0&0&\frac12&0&0&0&0&0\\
c&0&0&0&1&0&0&0&0\\
\frac{c-1}2&1&c&\frac{1-c}2&0&1&0&0&0\\
\frac c2&0&-c&-\frac c2&0&0&1&0&0\\
0&0&0&c&0&0&0&1&0\\
0&1&2c&0&0&0&0&0&1\\
\end{pmatrix}+o(1)\bigg)
$$
where $c=\cot\theta$.

So the local inversion at the boundary applies and the map $\widetilde H$ sends
diffeomorphically a neighborhood of the point $(D_{p,v,t},1,0,1,-1,0)$ in
$\widetilde\Sigma\times[-1,1]\times \D^2$ into a neighborhood of $(tp+\sqrt{1-t^2}v,tp+\sqrt{1-t^2}p\wedge
v,tp-\sqrt{1-t^2}v)$ in
$\B^3$ when $t$ is close to $1$. This implies that when $(q_1,q_2,q_3)\in
\inter\B^3$ is close to $(tp+\sqrt{1-t^2}v,tp+\sqrt{1-t^2}p\wedge
v,tp-\sqrt{1-t^2}v)$ ($t$ close to $1$) there is a unique $(\sigma,x_1,x_2,y_2,x_3,y_3)\in
\widetilde\Sigma\times(-1,1)\times \inter\D^2$ such that $\widetilde
H(\sigma,x_1,x_2,y_2,x_3,y_3)=(q_1,q_2,q_3)$ and the differential of $\widetilde
H$ at this point is invertible. So $\widetilde H$ has mod 2 degree equal to $1$. 
\end{proof}

Since $\widetilde L$ has odd degree, it is surjective. This finishes the proof
of Theorem~\ref{th:3point}.


\section{Minimal planes in asymptotically flat $3$-manifolds}

\label{sec:extent}


In this section we prove the following result which is an extension of Theorem~1 in~\cite{ChKe}.

\begin{thm}\label{th:minplane}
Let $(M,g)$ be an asymptotically flat $3$-manifold containing no closed embedded
minimal surface. Let us consider either
\begin{itemize}
\item $q\in M$ and $V\subset T_qM$ a $2$-dimensional subspace or
\item $q_1,q_2,q_3\in M$.
\end{itemize}
Then there is a complete properly embedded minimal plane $\Sigma$ in $M$ satisfying respectively
\begin{itemize}
\item $q\in \Sigma$ and $T_q\Sigma=V$ or
\item $q_i\in \sigma$ for $1\le i\le 3$.
\end{itemize}
\end{thm}

First let us recall what is the asymptotic flatness hypothesis we consider. $M$
is diffeomorphic to $\R^3$ and, in these coordinates, the metric can be written
$g=g_0+h$ where $g_0$ is the Euclidean metric and $h$ satisfies
$|h|+r|Dh|+r^2|D^2h|\to 0$ as $r\to\infty$ ($r$ the Euclidean distance to the
origin of $\R^3$).

Actually, once Theorems~\ref{th:prescrip} and \ref{th:3point} are known, the proof of the above
theorem follows the ideas of Chodosh and Ketover in~\cite{ChKe}. More precisely
we use a variant of \cite[Proposition~7]{ChKe}. Before let us fix a
notation. Let $\{\Gamma_R\}$ be a family of closed curves in $\S^2(R)=\partial
\B(R)$ where $\B(R)$ is the Euclidean ball of radius $R$. We say that they are
$C^{2,\alpha}$ almost parallel curves if, after
rescaling to unit size, $\frac1R\Gamma_R\subset  \S^2$ converges in the $C^{2,\alpha}$ sense to a
parallel curve in $\S^2$ or to a constant map.

\begin{prop}\label{prop:compact}
Let $(M,g)$ be an asymptotically flat manifold diffeomorphic to $\R^3$ that
contains no closed embedded minimal surfaces. 

Let $\{\Sigma_R\}$ be a
family of embedded minimal disks in
$\B(R)$ containing $q\in M$ and whose boundaries $\partial\Sigma_R\subset
\partial \B(R)$ are $C^{2,\alpha}$ almost parallel curves. Then a subsequence of
$\{\Sigma_R\}$ converges smoothly on compact subsets of $M$ to a complete
properly embedded minimal plane $\Sigma_\infty$.
\end{prop}

The difference with \cite{ChKe} is that we do not assume \textit{a priori} that
$\partial\Sigma_R$ is almost equatorial.

\begin{proof}
First we notice that, for large $R$, the
length of $\partial \Sigma_R$ is less than $2\pi(1+o(1)) R$.

Let $\lambda_R\simeq \sqrt R$ be such that $\Sigma_R\cap \S^2(\lambda_R)$ is
transverse. As in \cite{ChKe}, $\text{area}(\Sigma_R\setminus \B
^3(\lambda_R))\le \pi(1+o(1))R^2$. So $\widetilde\Sigma_R=\frac1R
(\Sigma_R\setminus \B^3(\lambda_R))$ is a stationary integral varifold in
$(\B,\tilde g_R)$ ($\tilde g_R$ is the homothetic of $g$ by $\frac1R$ and
converges to the Euclidean metric as $R\to \infty$).

Since $\|\widetilde\Sigma_R\|(\B)\le \pi(1+o(1))$, we can assume that
$\widetilde\Sigma_R$ converge as varifolds to $V$ which is stationary in
$\B^3\setminus \{0\}$ endowed with the flat metric and satisfies
$\|V\|(\B\setminus\{0\})\le \pi$. Actually,
$V$ extends to a stationary varifold of $\B$ with $\|V\|(\B)\le \pi$. The
origin is in the support of $V$ since the origin is contained in $\Sigma_R$.

By the monotonicity formula $V$ is the varifold associated to a flat unit disk
through the origin with multiplicity one. This implies that $\partial \Sigma_R$
is converging to an equator. Now the rest of the proof is similar to the one of
\cite[Proposition~7.2]{ChKe}.
\end{proof}

\begin{proof}[Proof of Theorem \ref{th:minplane}]
Let $q$ and $V$  or $q_1,q_2,q_3$ be as in the statement. Since the metric is asymptotically
Euclidean, we consider a chart $M\simeq(\R ^3,g)$ with the prescribed asymptotics
for $g$. This implies that for large $R$, the ball $\B(R)$ is mean convex. So we
can apply Theorem~\ref{th:prescrip} or \ref{th:3point} and obtain a minimal disk $\Sigma_R$ whose
boundary is an almost parallel curve such that $(q,V)=T_q\Sigma_R$ or $q_i\in \Sigma_R$.

By Proposition~\ref{prop:compact}, a subsequence of $\Sigma_R$ converges to
$\Sigma_\infty$ a properly embedded minimal plane in $M$. Since the convergence
is smooth $(q,V)=T_q\Sigma_\infty$ or $q_i\in \Sigma_\infty$.
\end{proof}

\begin{remarq}
As in \cite{ChKe}, Theorem~\ref{th:minplane} extends to the case where $M$ is
asymptoticaly conical: $M$ is $\R^3$ endowed outside a compact subset with a
metric $g=g_\alpha+h$ where $g_\alpha$ is the conical metric
$g_\alpha=dr^2+r^2\alpha^2g_{\S^2}$ and $h$ satisfies
$|h|+r|Dh|_{g_\alpha}+r^2|D^2h|_{g_\alpha}\to 0$ as $r\to\infty$
\end{remarq}


\section{A counterexample to uniqueness}


The arguments used in the preceding section prove that the minimal plane $\Sigma$ 
given by Theorem~\ref{th:minplane} for some $(q,V)$ (or $(q_1,q_2,q_3$) satisfies 
$|\Sigma\cap \B^3(R)|\le \pi(1+o(1))R^2$. One could ask the question of the uniqueness 
of such a minimal plane with quadratic area growth. In this section we give a 
counterexample to such a statement.

More precisely, we are going to construct on $\R^3$, an asymptotically Euclidean metric 
with no closed minimal surface
and $2$ (actually $3$) distinct minimal planes tangent at the origin 
that have quadratic area growth.

In order to construct this metric, let us consider a $C^\infty$ radial function $\phi : 
\R^3\to \R$  whose support is the unit ball and such that $0\le\phi\le1$ and $\phi=1$ on 
the ball of radius$\frac9{10}$.

Let us consider $r,\eps>0$ and define the function $\psi_{r,\eps}$ by
$$
\psi_{r,\eps}(x,y,z)=\sum\eps^2\phi\left((x,y,\frac z\eps)-(1+r)(\pm e_1+\pm
e_2)\right)
$$
where $e_1, e_2$ are the horizontal vectors of the canonical basis of $\R^3$.
We notice that the support of $\psi_{r,\eps}$ is the union of four ellipsoids
centered at $(1+r)(\pm e_1+\pm e_2)$. Moreover, $\psi_{r,\eps}$ is symmetric
with respect to the coordinate planes and axis. Finally, as $\eps\to 0$,
$\psi_{r,\eps}\to 0$ in a $C^1$ sense.

On $\R^3$, we consider the metric $g_{r,\eps}=(1+\psi_{r,\eps})^2 g_0$ which is
Euclidean outside a large ball. We are going to prove that, for a good choice
of $(r,\eps)$,  $g_{r,\eps}$ gives a counterexample to the uniqueness.

First we solve a Plateau problem. We notice that the domain
$\Ome_R$ of $\R^3$ defined by $\{(x,y,z)\in \R^3\mid  x\ge 0,y\ge 0, x^2+y^2\le
R^2,0\le z\le \eps\}$ ($R\ge (\sqrt2+1)(1+r)$) is mean convex. In $\partial\Ome_R$,
we consider the curve $\Gamma_R=\partial(\Ome_R\cap\{z=0\})$. Let us
study the Plateau problem associated to $\Gamma_R$ in $\Ome_R$. One possible
candidate solution is given by the quarter disk $\Delta_R=\partial\Ome_R\cap\{z=0\}$
which is minimal. For the metric $g_{r,\eps}$, its area can be bounded below by
$$
|\Delta_R|\ge \frac14\pi
R^2-\pi+\pi(1-(\frac9{10})^2)+\pi(1+\eps^2)^2(\frac9{10})^2=\frac14\pi
R^2+2\pi(\frac9{10})^2\eps^2+o(\eps^2)
$$

Let us construct an other competitor for the plateau problem. Let $\rho$ be the
radial Euclidean distance in $\R^2$ from the point $(1+r)(\pm e_1+\pm e_2)$ and
consider the radial function $u$ in $\R^2$ defined by 
$$
u(\rho)=\begin{cases}
\eps& \text{if }\rho\le 1\\
-\eps\frac{\ln\frac{\rho}{1+r}}{\ln(1+r)}& \text{if }1\le \rho\le 1+r\\
0& \text{if }\rho\ge 1+r
\end{cases}
$$
We notice that $0\le u\le \eps$ and if $G$ is the intersection of the graph of
$u$ with $\Ome_R$, we have $\partial G=\Gamma_R$ (we use here that $R\ge 
(\sqrt2+1)(1+r)$). Moreover $G$ is contained in
the part of $\Ome_R$ where $g_{r,\eps}=g_0$. So we can compute its area by
\begin{align*}
|G|&= \frac14\pi R^2-\pi(1+r)^2+\pi+2\pi\int_1^{1+r}\sqrt{1+u'^2(\rho)}\rho
d\rho\\
&\le \frac14\pi
R^2-\pi(1+r)^2+\pi+2\pi\int_1^{1+r}(1+\frac{\eps^2}{\ln^2(1+r)\rho^2})\rho
d\rho\\
&\le \frac14\pi R^2+2\pi \frac{\eps^2}{\ln(1+r)}
\end{align*}
So by choosing $r$ large enough in order to have
$\frac{2\pi}{\ln(1+r)}<2\pi(\frac9{10})^2$ we are sure that for small $\eps$,
$|G|\le |\Delta_R|$. Hence the solution to the Plateau problem is not
$\Delta_R$.

So let us fix $r$ as above and consider an increasing sequence $(R_n)$
converging to $+\infty$ with $R_n\ge (\sqrt 2+1)(1+r)$. Let us choose $\eps$ as above 
for $R_1$. Let $D_1$ be the solution of the Plateau problem in
$\Ome_{R_1}$. Since $D_1$ is not $\Delta_{R_1}$, $D_1$ touches $\Delta_1$ only
on its boundary. Now let $\Ome_{R_n}'$ be the part of $\Ome_{R_n}$ above $D_1$.
$\Ome_{R_n}'$ is also mean convex and we can consider $D_n$ the solution of the
Plateau problem in $\Ome_{R_n}'$ for the boundary curve $\Gamma_{R_n}$. The
sequence of minimal surfaces $(D_n)$ has uniformly bounded curvature since they are
stable. Moreover they are area minimizing so they have local uniform area
bounds. As a consequence, considering a subsequence, $D_n\to D$ where $D$ is a
minimal surface bounded by two halflines $\R_+\times\{(0,0)\}$ and
$\{0\}\times\R_+\times\{0\}$. After symmetry along its boundary we get a minimal
plane $P$ whose tangent space at the origin is horizontal. Since $P$ is above
$D_1$, $P$ is not $\{z=0\}$. Since $D_n$ is area minimizing, the area of
$P$ grows like $\pi R^2$. Finally, $P$, the plane $\{z=0\}$ and the symmetric of $P$ with 
respect to $\{z=0\}$ are three distinct minimal planes with horizontal tangent space at the 
origin.

The last point to verify is the non-existence of a compact minimal surface in
$(\R^3,g_{r,\eps})$. Since the metric $g_{r,\eps}$ converges to the Euclidean
one as $\eps\to 0$ and is Euclidean outside the ball of radius $\sqrt 2(2+r)$,
the spheres $\{x^2+y^2+z^2=R^2\}$ are mean convex for any $R>0$. So by the maximum 
principle, no compact minimal surface can exist in $(\R^3,g_{r,\eps})$. The $C^1$
convergence of the metric is sufficient to control the mean curvature of the
spheres.


\appendix

\section{About the Smale theorem}\label{appendix}


Let us explain the use of Smale's theorem in Lemma~\ref{lem:smale}. \textit{A
priori} this theorem says that the group $\diff_\infty(\D,\partial \D)$ is
contractible. But here we are not considering $C^\infty$ diffeomorphisms: they
are $C^{k,\alpha}$.

Let us prove a first result about the regularity of the diffeomorphisms we
consider. We focus on the case $k=2$ since this is enough for us.
\begin{lem}
Let $F,G\in C^{2,\alpha}(\D,\R^3)$ be two immersions and $\phi\in
\diff_1(\D,\partial \D)$ such that $F=G\circ \phi$. Then $\phi\in
\diff_{2,\alpha}(\D,\partial \D)$.
\end{lem}
\begin{proof}
Let $\sigma$ be the surface parametrized by $F$ and
$G$. It is well known that $\phi$ is actually a $C^2$
diffeomorphism. If we look at the second differential of $F=G\circ \phi$, we
have
$$
D^2F_{|p}=D^2G_{|\phi(p)}(D\phi_{|p},D\phi_{|p})+DG_{|\phi(p)}(D^2\phi_{|p})
$$
Thus 
$$
D^2\phi_{|p}=DG_{|\phi(p)}^{-1}\big(D^2F_{|p}-D^2G_{|\phi(p)}(D\phi_{|p},D\phi_{|p})\big)
$$
Let $\pi_p$ be the orthogonal projection from $\R^3$ to $T_{F(p)}\sigma$ and
$H_p:T_p\D\times T_p\D\to T_{F(p)}\sigma\subset \R^3$ defined by
$H_p=D^2F_{|p}-D^2G_{|\phi(p)}(D\phi_{|p},D\phi_{|p})$. Thus we have
\begin{align*}
D^2\phi_{|p}-D^2\phi_{|q}&=DG_{|\phi(p)}^{-1}H_p-DG_{|\phi(q)}^{-1}H_q\\
&=DG_{|\phi(p)}^{-1}H_p-DG_{|\phi(p)}^{-1}(\pi_p H_q)+DG_{|\phi(p)}^{-1}(\pi_p H_q)-DG_{|\phi(q)}^{-1}H_q\\
&=DG_{|\phi(p)}^{-1}\circ \pi_p(H_p- H_q)+(DG_{|\phi(p)}^{-1}\circ\pi_p -DG_{|\phi(q)}^{-1}\circ \pi_q)H_q
\end{align*}
Since $F$ and $G$ are $C^{2,\alpha}$, $\|H_p-H_q\|\le C|p-q|^\alpha$ and
$\|DG_{|\phi(p)}^{-1}\circ\pi_p -DG_{|\phi(q)}^{-1}\circ \pi_q\|\le C|p-q|$. As
a consequence, we have $\|D^2\phi_{|p}-D^2\phi_{|q}\|\le C|p-q| ^\alpha$. So
$\phi\in C^{2,\alpha}(\D,\D)$ and $\phi^{-1}$ also.
\end{proof}

Now let us explain what is the consequence of Smale's theorem : the contractibility of
$\diff_\infty(\D,\partial \D)$. Let $\barre \phi=(\bar a,\bar b)$ be in
$\diff_{k,\alpha}(\D,\partial \D)$. Let $a_t$ and
$b_t$  be the two solutions of the heat equations:
$$
\begin{cases}
\partial_t a_t=\Delta a& \text{on }\R_+^*\times\D\\
a_0=\bar a&\text{on }\D\\
a_t(x,y)=x&\text{on }\R_+\times \partial \D
\end{cases}
\quad \text{and}\quad 
\begin{cases}
\partial_t b_t=\Delta b& \text{on }\R_+^*\times\D\\
b_0=\bar b&\text{on }\D\\
b_t(x,y)=y&\text{on }\R_+\times \partial \D
\end{cases}
$$
These solutions exist, are unique and $a,b\in C^k(\R_+\times \D)\cap
C^\infty(\R_+^*\times \D)$. We define $\phi_t=(a_t,b_t)$.

First we remark that $\|\phi_t\|^2=a_t^2+b_t^2$ satisfies $\partial_t \|\phi_t\|^2\le
\Delta \|\phi_t\|^2$. From the maximum principle, $\phi_t(p)\in \D$ for all
$p\in \D$ and $t\ge 0$. Thus for $t$ close to $0$, $\phi_t\in
\diff_\infty(\D,\partial \D)$ since its Jacobian does not vanish. 

Let $Z$ be a
map from a compact set to $\diff_{k,\alpha}(\D,\partial \D)$ continuous in the $C^k$
topology. By solving
the heat equation as above, we construct a map $t\mapsto Z_t\in
\diff_{k,\alpha}(\D,\partial \D)$ for $t\in [0,\eps]$ with $Z_0=Z$ and $Z_t\in
\diff_\infty(\D,\partial \D)$ if $t>0$. This map is continuous in the $C^k$
topology. Since $Z_\eps$ is a continuous map with values in
$\diff_\infty(\D,\partial \D)$ which is contractible, we can deform it to the
constant map.

So $Z$ can be deformed in the $C^k$ topology to the constant map. As a consequence,
any homotopy group of $\diff_{k,\alpha}(\D,\partial D)$ is trivial. This is
sufficient for the proof of Lemma~\ref{lem:smale}.

\bibliographystyle{amsplain}
\bibliography{../reference.bib}

\end{document}